\documentclass[twoside]{amsart}
\usepackage{mathrsfs}
\usepackage{amsfonts}
\usepackage{amssymb}
\usepackage{amssymb}
\usepackage{amssymb}
\usepackage{amssymb}
\usepackage{amssymb}
\usepackage{amssymb}
\usepackage{amssymb}
\usepackage{amssymb}
\usepackage{amsmath}
\usepackage{mathdots}
\usepackage{xcolor}
\usepackage[all]{xy}

\usepackage{lastpage}

\RequirePackage{amsmath} \RequirePackage{amssymb}
\usepackage{amscd,latexsym,amsthm,amsfonts,amssymb,amsmath,amsxtra}
\usepackage[colorlinks=true, urlcolor=blue, linkcolor=red,anchorcolor=blue,citecolor=blue]{hyper ref}

    \newcommand{\BC}{{\mathbb {C}}}

     \newcommand{\BZ}{{\mathbb {Z}}}

     \newcommand{\CB}{{\mathcal {B}}}
     
     \newcommand{\CF}{{\mathcal {F}}}
     
     \newcommand{\CJ}{{\mathcal {J}}}
     
    \newcommand{\CM}{{\mathcal {M}}} \newcommand{\CN}{{\mathcal {N}}}

    \newcommand{\CS}{{\mathcal {S}}} 
    \newcommand{\CU}{{\mathcal {U}}} 
    \newcommand{\CW}{{\mathcal {W}}}

     \newcommand{\fo}{{\mathfrak{o}}}  \newcommand{\fp}{{\mathfrak{p}}}

    \newcommand{\lenth}{{\mathrm {\lenth}}}

    \newcommand{\Hom}{{\mathrm{Hom}}} 
    \newcommand{\Ind}{{\mathrm{Ind}}}

    \renewcommand{\Re}{{\mathrm{Re}}}

    \newcommand{\wt}{\widetilde}

    \newcommand{\ov}{\overline}

    \theoremstyle{plain}

    \newtheorem{thm}{Theorem}[section] \newtheorem{corollary}[thm]{Corollary}
    \newtheorem{lemma}[thm]{Lemma}  \newtheorem{proposition}[thm]{Proposition}
     \newtheorem{definition}[thm]{Definition}

    \numberwithin{equation}{section}

\usepackage[top=1in, bottom=1in, left=1.25in, right=1.25in]{geometry}

\begin{document}
\title{Asai gamma factors over finite fields}

\begin{abstract}
In this note, we define and study Asai gamma factors over finite fields. We also prove some results about local Asai L-functions over p-adic fields for level zero representations.
\end{abstract}

	\author{Jingsong Chai}
	\address{School of Mathematics, Physics and Finance \\ Anhui Polytechnic University \\Wuhu, Anhui,  241000\\China}
	\email{jingsongchai@hotmail.com}

	\thanks{The author is supported by a start up funding of AHPU}

	\maketitle
\section{Introduction}

In representation theory of connected linear algebraic group over local and finite fields, it is very useful and important to attach local invariants to irreducible representations, which encode various information about the representations in question. For example, this includes the work of Jacquet,Piatetski-Shapiro and Shalika (\cite{JPSS}) and Shahidi (\cite{Sh}).

Over finite fields, in her master thesis \cite{R}, Roditty-Gershon defined a finite field analog of the tensor gamma factor of Jacquet-Piatetski-Shapiro-Shlika. These analogs were then used by Nien to prove a finite field analogue of Jacquet's conjecture in \cite{N}. Later, Rongqing Ye in \cite{Y} showed that these tensor gamma factors are related to their local field counterpart through level zero supercuspidal representations.

This paper is inspired by the work of Ye \cite{Y}, and we consider the Asai gamma factors here. To be more precise, let $F=F_q$ be a finite field with $q$ elements, and $E/F$ is a quadratic extension so that $E=F_{q^2}$. Let $\pi$ be an irreducible cuspidal representation of $GL_n(E)$ with Whittaker model $\CW(\pi,\psi)$. Choose a function $\varphi$ on $F^n$, define the sum
\[
Z(W,\varphi;\psi)=\sum_{g\in N_n(F)\backslash GL_n(F)}W(g)\varphi(e_ng)
\]
where $e_n=(0,...,0,1)\in F^n$. This sum $Z(W,\varphi;\psi)$ admits a functional equation, which defines the Asai gamma factor $\gamma(\pi,\psi, As)$.
\begin{thm}
  Suppose that $\pi$ is irreducible cuspidal and not distinguished with respect to $GL_n(F)$. Then for any $\varphi\in \CS(F^n)$, and any $W\in \CW(\pi, \psi_E)$, we have
\[
\gamma(\pi,\psi, As)Z(W,\varphi; \psi)=Z(\widetilde{W}, \CF_\psi\varphi; \psi^{-1}).
\]
\end{thm}

We next compute the absolute value of this $\gamma(\pi,\psi, As)$.

\begin{thm}
Let $\pi$ be an irreducible cuspidal representation of $GL_n(E)$.

(1). If $\pi$ is not distinguished with respect to $GL_n(F)$, then
\[
|\gamma(\pi,\psi,As )|=q^{\frac{n}{2}}.
\]

(2). If $\pi$ is distinguished with respect to $GL_n(F)$, then
\[
\gamma(\pi,\psi,As )=-1.
\]
\end{thm}

We also relate $\gamma(\pi,\psi,As)$ to that of level zero representations, see Proposition \ref{Equality} for a precise statement.

\section{Notations and Preliminaries}

Let $L/K$ be a unramified quadratic extensions of p-adic fields. We denote by $\theta$ the nontrivial Galois element of this field extension. We choose $L/K$ so that the residue field of $K$ is a finite field $F=F_q$ with $q$ elements, and then the residue field of $L$ is $E=F_{q^2}$, which is a quadratic extension of $F$. Let $\fo_L, \fo_K$ be ring of integers of $L,K$, with maximal ideals $\fp_L, \fp_K$, respectively. Choose a uniformizer $\varpi$ for both $\fo_L$ and $\fo_K$, as $L/K$ is unramified. Use $|\cdot|_L$ and $|\cdot|_K$ to denote their respective absolute values on $L$ and $K$. They are normalized so that $|\varpi|_L=q^{-2}$ and $|\varpi|_K=q^{-1}$. We choose a nontrivial additive character $\psi_K$ of $K$, which has conductor precisely $\fp_K$. This character $\psi_K$ descends to a nontrivial character of $F=\fo_K/\fp_K$. We will use $\psi$ to denote both these two characters when there is no confusion. Choose some $z\in \fo_L^\times \backslash \fo_K^\times$ so that $z-\theta(z)\in \fo_L^\times$, and set
\[
\psi_L(x)=\psi(\frac{x-\theta(x)}{z-\theta(z)}).
\]
One then can check that $\psi_L$ has conductor precisely $\fp_L$, and it descends to a nontrivial character $\psi_E$ of $E=\fo_L/\fp_L$. It then follows that $\psi_L$ is trivial on $K$ and $\psi_E$ is trivial on $F$.

If $R$ is a commutative ring, we use $N_n(R)$ to denote the standard $n\times n$ unipotent matrices with coefficients in $R$. If $\psi_R:R\to \BC^\times$ is a nontrivial additive character of $R$, we extend it to $N_n(R)$ by
\[
\psi_R(u)=\psi_R(\sum_{i=1}^{n-1} u_{i,i+1}).
\]
for $u=(u_{ij})\in N_n(R)$, and still denote it by $\psi_R$.

Set $G=GL_n(E), H=GL_n(F)$. Let $\pi$ be an irreducible representation $\pi$ of $G$, we say $\pi$ is generic if $\Hom_G(\pi, Ind_{N_n(E)}^{G}\psi_E)\neq 0$. It is well known that the dimension of this $\Hom$ space is at most one, and when $\pi$ is generic, it is one. Then by Frobenius reciprocity law, $\dim \Hom_{N_n(E)}(\pi|_{N_n(E)}, \psi_E)=1$. Let $l\in \Hom_{N_n(E)}(\pi|_{N_n(E)}, \psi_E)$ be a nonzero Whittaker functional of $\pi$. Then for $v\in V_\pi$, define $W_v(g):=l(\pi(g)v)$, which is called the Whittaker function attached to the vector $v$. The space generated by all Whittaker functions $W_v(g)$ is called the Whittaker model $\CW(\pi,\psi_E)$ of $\pi$.

For a Whittaker function $W_v\in \CW(\pi,\psi_E)$, define a function $\wt{W}_v$ on $G$ by
\[
\wt{W}_v(g)=W_v(\omega_n ({^t}g^{-1}))
\]
where $\omega_n$ is the longest Weyl element of $G$, with 1's on the anti-diagonal and zero elsewhere. Then the function $\wt{W}_v(g)\in \CW(\wt{\pi},\psi_E^{-1})$, where $\wt{\pi}$ denotes the contragredient representation of $\pi$.

An irreducible representation $\pi$ of $G$ is called cuspidal if it has no $N_n(E)$-fixed vectors for any unipotent radical $N(E)$ of proper standard parabolic subgroups.

\begin{definition}
Let $\pi$ be an irreducible generic representation of $G$. We call a function $\CB_{\pi,\psi_E}$ a Bessel function with respect to $(\pi,\psi_E)$, if $\CB_{\pi,\psi_E}$ is a Whittaker function and satisfies
\[
\CB_{\pi,\psi_E}(u_1gu_2)=\psi_E(u_1u_2)\CB_{\pi,\psi_E}(g)
\]
for all $g\in G, u_1,u_2\in N_n(E)$. We will normalize $\CB_{\pi,\psi_E}$ so that it take value 1 at the identity.
\end{definition}

The existence and uniqueness of Bessel functions are guaranteed by Proposition 4.2 and 4.3 in \cite{Gel}. It also has the following explicit expression. Let $\chi_\pi$ be the character of $\pi$, then for all $g\in G$,
\[
\CB_{\pi,\psi_E}(g)=|N_n(E)|^{-1} \sum_{u\in N_n(E)} \psi_E^{-1}(u)\chi_\pi(gu).
\]

By Proposition 2.15 and 3.5 in \cite{N}, Bessel function satisfies the following identity
\[
\CB_{\pi,\psi_E}(g^{-1})=\overline{\CB_{\pi,\psi_E}(g) }=\CB_{\wt{\pi},\psi_E^{-1}}(g).
\]

Now we turn to p-adic field and recall the local Asai gamma factor defined by A.Kable in \cite{K}. We shall identify $K^n$ with the space of row vectors of length $n$. Consider a Schwartz function $\varphi$ on $K^n$, we define its Fourier transform as
\[
\CF_\psi\varphi(y)=\int_{K^n} \varphi(x)\psi(x{^t}y)dx.
\]

Let $\sigma$ be a smooth irreducible generic representation of $GL_n(L)$ with Whittaker model $\CW(\sigma, \psi_L)$ with respect to the character $\psi_L$. For $s\in \BC$, $\varphi$ a Schwartz function on $K^n$, and $W\in \CW(\sigma, \psi_L)$, we define the integral
\[
Z(W,\varphi,s)=\int_{N_n(K)\backslash GL_n(K)} W(g)\varphi(e_n g)|\det g|_K^s dg
\]
where $e_n=(0,...,0,1)$.
\begin{thm}{(A.Kable)}
\label{LocalAsai} (1). The integral $Z(W,\varphi, s)$ converges absolutely when $\Re(s)$ is sufficiently large. The subspace of $\BC(q^{-s})$ spanned by the local integrals $Z(W,\varphi,s)$ is a $\BC[q^s,q^{-s}]$-fractional ideal and contains $1$.

(2). There exists a function $\gamma(s,\sigma, \psi, As)\in \BC(q^{-s})$ such that
\[
\gamma(s,\sigma,\psi,As)Z(W,\varphi,s)=Z(\wt{W}, \CF_\psi\varphi, 1-s)
\]
for all Schwartz functions $\varphi$ and all Whittaker functions $W\in \CW(\sigma,\psi_L)$.
\end{thm}

We can use this theorem to define local Asai L-factors and epsilon factors. Note that $\BC[q^{-s},q^s]$ is a PID and hence the fractional ideal is principal. Since this ideal contains $1$, we may choose a unique generator of the form $1/P_\sigma(q^{-s})$, where $P_\sigma(T)\in \BC[T]$ and $P(0)=1$. Then we define
\[
L(s,\sigma, As)=\frac{1}{P_\sigma(q^{-s})}.
\]
This is the local Asai L-factor, and the local epsilon factor is defined in the usual way as
\[
\varepsilon(s,\sigma,\psi, As)=\gamma(s,\sigma,\psi,As)\frac{L(s,\sigma,As)}{L(1-s, \wt{\sigma},As)}.
\]

\section{Asai gamma factors}
Assume $\pi$ is generic with Whittaker model $\CW(\pi,\psi_E)$. Let $\CS(F^n)$ be the space of functions $\varphi:F^n\to \BC$. For a function $\varphi\in \CS(F^n)$, we define its Fourier transform $\CF_\psi\varphi:F^n \to \BC$ by the formula
\[
\CF_\psi\varphi(y)=\sum_{x\in F^n}\varphi(x)\psi(<x,y>),
\]
where if $x=(x_1,...,x_n)\in F^n$ and $y=(y_1,...,y_n)\in F^n$, then $<x,y>$ is the standard pairing
\[
<x,y>=\sum_{i=1}^n x_iy_i.
\]
Take a Whittaker function $W\in \CW(\pi,\psi)$, and a function $\varphi$ on $F^n$. Consider the following sum
\[
Z(W,\varphi;\psi)=\sum_{g\in N_n(F)\backslash H}W(g)\varphi(e_ng).
\]
The following theorem establishes the functional equation for these sums and defines Asai gamma factor for cuspidal representations. Its proof is imitated from Theorem 2.3 in \cite{Y} with necessary modifications.

\begin{thm}
\label{AsaiGamma} If $\pi$ is an irreducible cuspidal representation of $G$, then there exists a nonzero constant $\gamma(\pi,\psi, As)$, such that for any $W\in \CW(\pi,\psi_E)$, and any $\varphi\in \CS(F^n)$ with $\varphi(0)=0$, we have
\[
\gamma(\pi,\psi, As)Z(W,\varphi; \psi)=Z(\widetilde{W}, \CF_\psi\varphi; \psi^{-1}).
\]
\end{thm}
\begin{proof}
Let $\CS_0(F^n)$ be the set of functions $\varphi$ on $F^n$ such that $\varphi(0)=0$. The group $H$ acts on $\CS_0(F^n)$ by right multiplication, i.e.,
\[
(R(g)\varphi)(x)=\varphi(xg)
\]
for $g\in H$ and $\varphi\in \CS_0(F^n)$. If we set
\[
L_1(W,\varphi)=\sum_{g\in N_n(F)\backslash H} W(g)\varphi(e_ng)
\]
and
\[
L_2(W,\varphi)=\sum_{g\in N_n(F)\backslash H} \widetilde{W}(g)\CF_\psi\varphi(e_ng),
\]
we can check that both $L_1$ and $L_2$ belong to the space $\Hom_H(\pi\otimes \CS_0(F^n), 1)$.  We are going to show that this space has dimension one. Set $P_n$ to be the mirabolic subgroup consisting of matrices with last row being $(0,...,0,1)$.

\textit{Claim}: $\dim \Hom_H(\pi\otimes \CS_0(F^n), 1)=\dim \Hom_{P_n(F)}(\pi,1)$.

We first prove this Claim. Let $(v,\varphi)$ be an $H$-invariant bilinear form on $\pi\otimes \CS_0(F^n)$. Let $\phi_n$ be the indicator function of $e_n$, and consider a linear form
\[
l(v)=(v,\phi_n)
\]
on $\pi$. Since $\phi_n$ is fixed by the $P_n(F)$ and $l(v)$ is a $P_n(F)$-invariant linear form on $\pi$, $l\in  \Hom_{P_n(F)}(\pi,1)$. This implies that $\dim \Hom_H(\pi\otimes \CS_0(F^n), 1)\leqslant \dim \Hom_{P_n(F)}(\pi,1)$. Conversely, for any $\varphi\in \CS_0(F^n)$, since $H$ acts transitively on $F^n\backslash \{0\}$, we can write
\[
\varphi=\sum_{g\in H/P_n(F)} c_g R(g)\phi_n
\]
for some $c_g\in \BC$. We can then define $(v,\varphi)$ on $\pi\otimes \CS_0(F^n)$ as
\[
(v,\varphi)=\sum_{g\in H/P_n(F)} c_gl(g^{-1}v).
\]
It can be checked that $(v,\varphi)$ is $H$-invariant. Hence $\dim \Hom_H(\pi\otimes \CS_0(F^n), 1)\geqslant \dim \Hom_{P_n(F)}(\pi,1)$, and the Claim is proved.

Therefore, it suffices to prove that $\dim \Hom_{P_n(F)}(\pi,1)=1$. Since $\pi$ is a cuspidal representation of $G$, we have $\pi|_{P_n(E)}=\Ind_{N_n(E)}^{P_n(E)}\psi_E$ is irreducible, for example, by Theorem 2.3 in \cite{Gel}. Then by Proposition 4.3 in \cite{AM}, we have
\[
\Hom_{P_n(F)}(\Ind_{N_n(E)}^{P_n(E)}\psi_E, 1)\cong \BC
\]
which finishes the proof.
\end{proof}

Similar to the Rankin-Selberg case, we can express $\gamma(\pi,\psi,As)$ in terms of Bessel functions.

\begin{proposition}
\label{Bessel} If $\pi$ is irreducible cuspidal, then
\[
\gamma(\pi,\psi,As)= \sum_{g\in N_n(F)\backslash H} \CB_{\pi,\psi_E}(g)\psi(<e_ng^{-1}, e_1>),
\]
where $e_1=(1,0,...,0)\in F^n$.
\end{proposition}
\begin{proof}
We choose $\varphi$ to be $\phi_n$, the indicator function of $e_n$. Then one computes the Fourier transform
\[
\CF_\psi \phi_n(y)=\sum_{x\in F^n} \phi_n(x)\psi(<y,x>)=\psi(<y,e_n>)=\psi(y_n)
\]
if we write $y=(y_1,...,y_n)$.

In the functional equation
\[
\gamma(\pi,\psi, As)Z(W,\varphi; \psi)=Z(\widetilde{W}, \CF_\psi\varphi; \psi^{-1}),
\]
we have
\[
Z(\CB_{\pi,\psi_E},\varphi; \psi)=\sum_{g\in N_n(F)\backslash H} \CB_{\pi,\psi_E}(g)\phi_n(g)=\sum_{g\in N_n(F)\backslash P_n(F)} \CB_{\pi,\psi_E}(g).
\]
By Proposition 4.9 and its corollary in \cite{Gel}, when restricted to $P_n(F)$, $\CB_{\pi,\psi}(g)=0$ unless $g\in N_n(F)$, and thus this sum reduces to $1$.

For the right hand side, we have
\[
Z(\widetilde{\CB}_{\pi,\psi_E}, \CF_\psi\varphi; \psi^{-1})=\sum_{g\in N_n(F)\backslash H} \CB_{\pi,\psi_E}(g)\CF_\psi\phi_n(e_1 {^t}g^{-1})=\sum_{g\in N_n(F)\backslash H} \CB_{\pi,\psi_E}(g) \psi(g_{n1})
\]
if we write $g^{-1}=(g_{ij})$.

Direct computation shows that $\psi(<e_ng^{-1}, e_1>)=\psi(g_{n1})$ and the proposition follows.
\end{proof}

We note that
\begin{eqnarray*}
\ov{\gamma(\pi,\psi,As) }
&=& \sum_{g\in N_n(F)\backslash H }\ov{\CB_{\pi,\psi_E}(g)}\cdot \ov{\psi(<e_ng^{-1}, e_1>) }  \\
&=& \sum_{g\in N_n(F)\backslash H } \CB_{\wt{\pi},\psi_E^{-1}}(g) \psi^{-1}(<e_ng^{-1}, e_1>)  \\
&=& \gamma(\wt{\pi}, \psi^{-1}, As).
\end{eqnarray*}

We say $\pi$ is distinguished with respect to $H$ if $\Hom_H(\pi,1)\neq 0$. The next result extends the above functional equation to general functions $\varphi$ when $\pi$ is not distinguished with respect to $H$.

\begin{proposition}
\label{GeneralFE} Suppose that $\pi$ is irreducible cuspidal and not distinguished with respect to $H$. Then for any $\varphi\in \CS(F^n)$, and any $W\in \CW(\pi, \psi_E)$, we have
\[
\gamma(\pi,\psi, As)Z(W,\varphi; \psi)=Z(\widetilde{W}, \CF_\psi\varphi; \psi^{-1}).
\]
\end{proposition}
\begin{proof}
Set $\delta_0$ to be the indicator function of $0\in F^n$. First we note that
\[
Z(W,\delta_0; \psi)=\sum_{g\in N_n(F)\backslash H}W(g)\delta_0(e_ng)=0.
\]
Secondly, use $1$ to denote the constant function, then
\[
Z(W,1; \psi)=\sum_{g\in N_n(F)\backslash H}W(g)=0
\]
since this is an $H$-invariant linear form on $\pi$ and $\pi$ is not distinguished with respect to $H$.

Now for general $\varphi\in \CS(F^n)$, write $\varphi=\varphi_0 + \varphi_1$, where $\varphi_0=\varphi-\varphi(0)$ and $\varphi_1=\varphi(0)$. Then $\varphi_0(0)=0$ and one computes $\CF_\psi \varphi_1=q^n \varphi(0)\delta_0$. Apply the functional equation in Theorem \ref{AsaiGamma} to $\varphi_0$, and note that both $Z$ and $\CF_\psi$ are linear in $\varphi$, we find
\[
\gamma(\pi,\psi, As)(Z(W,\varphi; \psi)-Z(W, \varphi_1; \psi) )= Z(\widetilde{W}, \CF_\psi\varphi; \psi^{-1} )-Z(\widetilde{W}, \CF_\psi\varphi_1; \psi^{-1} ).
\]
By what we discussed above we find
\[
\gamma(\pi,\psi, As)Z(W,\varphi; \psi)=Z(\widetilde{W}, \CF_\psi\varphi; \psi^{-1}).
\]
\end{proof}

As a corollary, we find the absolute value of $\gamma(\pi,\psi, As)$.

\begin{corollary}
\label{AbsValue} Let $\pi$ be irreducible cuspidal and it is not distinguished with respect to $H$. Then
\[
\gamma(\pi,\psi, As)\gamma(\wt{\pi}, \psi^{-1},As )=q^n
\]
and therefore
\[
|\gamma(\pi,\psi,As )|=q^{\frac{n}{2}}.
\]
\end{corollary}
\begin{proof}
By the functional equation in Proposition \ref{GeneralFE},
\[
\gamma(\pi,\psi, As)Z(W,\varphi; \psi)=Z(\widetilde{W}, \CF_\psi\varphi; \psi^{-1}).
\]
Apply this to $\wt{\pi}$ again, we have
\[
\gamma(\wt{\pi},\psi^{-1},As )Z(\widetilde{W}, \CF_\psi\varphi; \psi^{-1})=Z(\wt{\wt{W} },\CF_{\psi^{-1}}\CF_{\psi}\varphi, \psi ).
\]
Now the corollary follows since $\CF_{\psi^{-1}}\CF_{\psi}\varphi=q^n\varphi$.
\end{proof}

For the case when $\pi$ is distinguished, we will use method in \cite{SZ} to compute $\gamma(\pi,\psi,As)$. In the next section, another method similar to that in \cite{Y} will be present. Before that, we first need the following lemma which characterizes the distinction of $\pi$ in terms of Bessel function $\CB_{\pi,\psi}$. First note that the Galois involution $\theta$ induces an involution of $G$ by componentwise, still denoted as $\theta$. Define $\tau:G\to G$ by
\[
\tau(g)={^t}(\theta(g))^{-1}.
\]
We also note that by the definition,
\[
\psi_E(\theta(u))=\psi_E(u^{-1})
\]
for all $u\in N_n(E)$.

\begin{lemma}
\label{BesselDistinction} Let $\pi$ be an irreducible cuspidal representation of $G$ with Bessel function $\CB_{\pi,\psi_E}$. Then $\pi$ is distinguished with respect to $H$ if and only if $\CB_{\pi,\psi_E}(\theta(g)^{-1})=\CB_{\pi,\psi_E}(g)$. In particular, if $\pi$ is distinguished, then for all $g\in H$, we have $\CB_{\pi,\psi_E}(g^{-1})=\CB_{\pi,\psi_E}(g)$.
\end{lemma}
\begin{proof}
Assume $\pi$ is distinguished, then by Theorem 3.6 in \cite{Gow}, $\chi_\pi(\tau(g))=\chi_\pi(g)$. Then we have
\begin{align*}
\CB_{\pi,\psi_E}(\theta(g)^{-1}) &= \frac{1}{|N_n(E)|} \sum_{u\in N_n(E)} \psi_E(u^{-1})\chi_\pi(\theta(g)^{-1} u ) \\
&= \frac{1}{|N_n(E)|} \sum_{u\in N_n(E)} \psi_E(u^{-1}) \chi_{\pi} ({^t}g \tau(u))  \\
&= \frac{1}{|N_n(E)|} \sum_{u\in N_n(E)} \psi_E(u^{-1}) \chi_\pi(\theta(u)^{-1}g )  \\
&= \frac{1}{|N_n(E)|} \sum_{u\in N_n(E)} \psi_E(u^{-1}) \chi_\pi(g\theta(u)^{-1} )  \\
&= \frac{1}{|N_n(E)|} \sum_{u\in N_n(E)} \psi_E(\theta(u)) \chi_\pi(gu)  \\
&= \frac{1}{|N_n(E)|} \sum_{u\in N_n(E)} \psi_E(u^{-1}) \chi_\pi(gu) \\
&= \CB_{\pi,\psi}(g).
\end{align*}
Conversely, if $\CB_{\pi,\psi_E}(\theta(g)^{-1})=\CB_{\pi,\psi_E}(g)$ for all $g\in G$, by Proposition 5.4 in \cite{Gel}, we have
\begin{align*}
\chi_\pi(\tau(g))&= \chi_\pi(\theta(g)^{-1})=\frac{1}{|N_{n-1}(E)|}\sum_{p\in P_n(E)}\CB_{\pi,\psi_E}(p\theta(g)^{-1}p^{-1})   \\
&= \frac{1}{|N_{n-1}(E)|}\sum_{p\in P_n(E)}\CB_{\pi,\psi_E}(\theta(p)^{-1}g\theta(p) )   \\
&= \frac{1}{|N_{n-1}(E)|}\sum_{p\in P_n(E)}\CB_{\pi,\psi_E}(pgp^{-1})  \\
&= \chi_\pi(g)
\end{align*}
and hence $\pi$ is distinguished by Theorem 3.6 in \cite{Gow} again.
\end{proof}

\begin{proposition}
\label{AbsValue1} Suppose $\pi$ is an irreducible cuspidal distinguished representation of $G$. Then
\[
\gamma(\pi,\psi,As)=-1.
\]
\end{proposition}
\begin{proof}
By Proposition 4.3 in \cite{AM}, since $\pi$ is cuspidal, we have
\[
\Hom_{P_n(F)}(\pi, 1)\cong \BC.
\]
On the other hand, by Theorem 3.6 in \cite{Gow},
\[
\Hom_{H} (\pi, 1) \cong \BC,
\]
So the containment $\Hom_{H} (\pi, 1) \subseteq \Hom_{P_n(F)} (\pi, 1)$ is in fact an equality.

Consider the linear form, for $v\in V_\pi$,
\[
l(v)=\sum_{p\in N_n(F)\backslash P_n(F)} W_v (p).
\]
This is nonzero and $l\in \Hom_{P_n(F)}(\pi, 1)$, which then implies that $l\in \Hom_{H} (\pi, 1)$. This means that $l(\pi(g)v)=l(v)$ for any $v$ and any $g\in H$. Apply this to $\CB_{\pi,\psi_E}$, and by Proposition 4.9 and its corollary in \cite{Gel}, we find
\[
\sum_{p\in N_n(F)\backslash P_n(F)}\CB_{\pi,\psi_E}(pg)=\sum_{p\in N_n(F)\backslash P_n(F)}\CB_{\pi,\psi_E}(p)=1.
\]

Now by Proposition \ref{Bessel} and Lemma \ref{BesselDistinction}, we get
\begin{align*}
\gamma(\pi,\psi,As)&= \sum_{g\in N_n(F)\backslash H} \CB_{\pi,\psi_E}(g)\psi(<e_ng^{-1}, e_1>) \\
&= \frac{1}{|N_n(F)|} \sum_{g\in H} \CB_{\pi,\psi_E}(g)\psi(<e_ng^{-1}, e_1>) \\
&= \frac{1}{|N_n(F)|} \sum_{g\in H} \CB_{\pi,\psi_E}(g^{-1})\psi(<e_ng^{-1}, e_1>)  \\
&= \frac{1}{|N_n(F)|} \sum_{g\in H} \CB_{\pi,\psi_E}(g)\psi(<e_ng, e_1>)  \\
&= \sum_{g\in N_n(F)\backslash H} \CB_{\pi,\psi_E}(g)\psi(<e_ng, e_1>)  \\
&= \sum_{g\in P_n(F)\backslash H} \sum_{p\in N_n(F)\backslash P_n(F)} \CB_{\pi,\psi_E}(pg)\psi(<e_ng, e_1>)   \\
&= \sum_{g\in P_n(F)\backslash H} \psi(<e_ng, e_1>).
\end{align*}
Now the conclusion follows from the computation in the proof of Lemma A.1 in \cite{SZ}.
\end{proof}

In the above proof, we find that, if $\pi$ is cuspidal distinguished, then
\[
\sum_{g\in N_n(F)\backslash H}\CB_{\pi,\psi_E}(g)=\sum_{P_n(F)\backslash H}1=\frac{|H|}{|P_n(F)|}=q^n-1.
\]
This will be used in the next section. We remark that there is a formula in \cite{AM}, see also \cite{Ya}, of such sum for general generic representations.

\section{Level zero representations}

A representation $\sigma$ of $GL_n(L)$ is of level zero if it is of the form
\[
\sigma \cong ind_{L^\times GL_n(\fo_L)}^{GL_n(L)} \Lambda
\]
where $\Lambda$ is a representation of $L^\times GL_n(\fo_L)$ such that $\Lambda |_{GL_n(\fo_L)}$ is an inflation of $\pi$ via $GL_n(\fo_L) \xrightarrow{mod \fp_L} GL_n(E)$. Here $\pi$ is an irreducible cuspidal representation of $GL_n(E)$, and $ind$ is the smooth compact induction.
Note that $L^\times\bigcap GL_n(\fo_L)\cong \fo_L^\times$, so $\Lambda$ is determined by $\pi$ and the nonzero complex number $\lambda=\omega_\Lambda(\varpi)$, where $\omega_\Lambda$ is the central character of $\Lambda$. Note that $\omega_\sigma=\omega_\Lambda$ in this case. By Theorem 8.4.1 in \cite{BK}, we then get a parametrization of level zero representations with the pair $(\lambda, \pi)$. We will just say a level zero representation $\sigma$ is from $(\lambda, \pi)$.

Recall that we have a unramified quadratic extension of p-adic fields $L/K$, with quadratic extension of residue fields $E/F$.
\begin{lemma}
\label{Lfactor} Let $\sigma$ be a level zero representation of $GL_n(L)$ coming from $(\lambda, \pi)$, where $\pi$ is a distinguished representation of $GL_n(E)$ with respect to $GL_n(F)$. Then
\[
L(s, \sigma, As)=(1-\lambda q^{-ns})^{-1}.
\]
\end{lemma}
\begin{proof}
By Proposition 3.6 in \cite{M},
\[
L(s, \sigma, As)=\prod (1-q^{s_0-s})^{-1}
\]
where the product is taken over the $q^{s_0}$'s such that $\sigma \otimes |\cdot|_L^{s_0/2}$ is distinguished. Now if $s_0$ is a such pole, by Lemma 2.1 in \cite{Y}, $\sigma \otimes |\cdot|_L^{s_0/2}$ is a level zero representation from $(\lambda q^{-ns_0} ,\pi)$. As this representation is distinguished, the restriction of its central character to $K^\times$ is trivial. Since the extension $L/K$ is unramified, $\varpi \in \fo_K$ and thus $\lambda q^{-ns_0}=1$.

Conversely, if $s_0$ is a complex number satisfying $\lambda q^{-ns_0}=1$. Consider the pair $(1,\pi)$. It determines a level zero representation $\sigma'\cong ind_{L^\times GL_n(\fo_L)}^{GL_n(L)} \Lambda'$. Set $G'=GL_n(L), H'=GL_n(K)$, by standard Mackey's theory as in section 5 of \cite{HK} for example, we have a decomposition
\[
\Hom_{H'}(\sigma', 1)= \bigoplus_{g\in L^\times GL_n(\fo_L)\backslash G'/H' } \Hom_{L^\times GL_n(\fo_L) \cap gH'g^{-1}  }(\Lambda', 1).
\]
Consider the direct summand corresponding to the identity double coset in the right hand side of the above equality, we have $L^\times GL_n(\fo_L) \cap H'=K^\times GL_n(\fo_K)$ and
\[
\Hom_{K^\times GL_n(\fo_K)}(\Lambda', 1) \cong \Hom_H(\pi, 1)\cong \BC.
\]
Since the pair $(G',H')$ is a Gelfand pair, it follows that other direct summands in the above decomposition are all zero and
\[
\Hom_{H'}(\sigma', 1)\cong \Hom_{K^\times GL_n(\fo_K)}(\Lambda', 1)\cong \BC,
\]
which then implies that $\sigma'$ is distinguished, and such $s_0$ is a pole of $L(s,\sigma,As)$.

Hence $s_0$ is a pole of $L(s,\sigma, As)$ if and only if $\lambda q^{-ns_0}=1$. There are $n$ such poles, and if we fix one pole $s_0$, these $n$ poles are
\[
s_0, s_0e^{\frac{2\pi i}{n\ln q}},...,s_0e^{\frac{2(n-1)\pi i}{n\ln q}}.
\]
Therefore,
\[
L(s, \sigma, As)= \prod_{k=0}^{n-1} (1-q^{s_0-s}e^{\frac{2\pi ik}{n} })^{-1}=(1-q^{ns_0-ns})^{-1}=(1-\lambda q^{-ns})^{-1},
\]
which finishes the proof.
\end{proof}

Suppose we are given a level zero representation $\sigma=ind_{L^\times GL_n(\fo_L)}^{GL_n(L)}\Lambda$ of $GL_n(L)$ corresponding to the pair $(\lambda, \pi)$. Denote $\bar{k}$ the image of $k\in GL_n(\fo_L)$ under the natural map $GL_n(\fo_L) \xrightarrow{mod \fp_L} GL_n(E) $. By Theorem 5.8 in \cite{PS}, there exists a Whittaker function $W_\sigma$ of $\sigma$ with respect to $\psi_L$ with the following properties:

(1). $Supp W_{\sigma} \subseteq N_n(L)L^\times GL_n(\fo_L)$ and
\[
W_\sigma(ug)=\psi_L(u)\CJ_\sigma (g)
\]
for any $y\in N_n(L), g\in L^\times GL_n(\fo_L)$,
where for $g=ak \in L^\times GL_n(\fo_L)$, with $a\in L^\times, k\in GL_n(\fo_L)$,
\[
\CJ_\sigma(g)=\omega_\sigma(a) |N_n(E)|^{-1} \sum_{h\in N_n(E)} \psi_E(h^{-1})\chi_\pi (\bar{k}h).
\]

(2). $Supp W_\sigma \bigcap P_n(L)= N_n(L)(H_n^1(L)\cap P_n(L) ) $ and for $u\in N_n(L)$ and $h\in H_n^1(L)\cap P_n(L)$,
\[
W_\sigma(uh)=\psi_L(u).
\]
Here $H_n^1(L)$ is the subgroup of $GL_n(\fo_L)$ consisting of matrices which will reduce to identity modulo $\fp_L$.

Let us recall results in section 4 and 5 in \cite{PS} to say more on the function $\CJ_\sigma$. Define a map $\Psi_L$ on $(N_n(L)\cap GL_n(\fo_L))H_n^1(L)$ by
\[
\Psi_L(uh)=\psi_L(u)
\]
for any $u\in N_n(L)\cap GL_n(\fo_L), h\in H_n^1(L)$. This is a character of the group $(N_n(L)\cap GL_n(\fo_L))H_n^1(L)$ since $H_n^1(L)$ is normal in $GL_n(\fo_L)$. Now by Theorem 4.4 in \cite{PS}, we know that the restriction $\Lambda |_{(P_n(L)\cap GL_n(\fo_L))H_n^1(L) }$ is irreducible and in fact we have
\[
\Lambda |_{\CM } \cong Ind_{\CU}^{\CM} \Psi_L.
\]
where $\CM= (P_n(L)\cap GL_n(\fo_L))H_n^1(L), \CU=(N_n(L)\cap GL_n(\fo_L))H_n^1(L)$.
To get the function $\CJ_\sigma$, take the group $\CN=(N_n(L)\cap GL_n(\fp_L) )H_n^1(L)$ in the kernel of $\Psi_L$, then, for $g\in GL_n(\fo_L)$,
\[
\CJ_\sigma(g)= (\CU:\CN)^{-1} \sum_{h\in \CU/\CN} \Psi_L(h^{-1})\chi_\Lambda (gh)=|N_n(E)|^{-1} \sum_{h\in N_n(E)} \psi_E(h^{-1})\chi_\pi (\bar{g}h).
\]

\begin{lemma}
\label{LHS} Assume $\sigma$ is of level zero, and let $W_\sigma$ be the Whittaker function of $\sigma$ defined as above. Set $\varphi$ to be the indicator function of $e_nH_n^1(K)$. Then, when restricted to $GL_n(K)$, the function
\[
W_\sigma(g)\varphi(e_n g)
\]
is the indicator function on the set $N_n(K)H_n^1(K)$, and
\[
Z(W,\varphi,s)=vol(N_n(K)\backslash N_n(K)H_n^1(K)).
\]
\end{lemma}
\begin{proof}
We first note that $Supp W_{\sigma} \subseteq N_n(L)L^\times GL_n(\fo_L)$ and $Supp \varphi(e_n g) = P_n(L)H_n^1(L)$. Thus the support of $W_\sigma(g)\varphi(e_n g)$ is contained in
\[
N_n(L)L^\times GL_n(\fo_L)\bigcap P_n(L)H_n^1(L)= N_n(L)(GL_n(\fo_L)\bigcap P_n(L) )H_n^1(L).
\]
Then, for any $u\in N_n(L), g\in (GL_n(\fo_L)\bigcap P_n(L) )H_n^1(L)$,
\[
W_\sigma(ug)\varphi(e_n ug)=\psi_L(u)\CJ_\sigma(g).
\]
Now by Proposition 5.3 in \cite{PS}, $\CJ_\sigma(g)=\Psi_L(g)$ if $g\in (GL_n(\fo_L)\bigcap N_n(L) )H_n^1(L)$, and $\CJ_\sigma(g)=0$ elsewhere. Now restrict to $GL_n(K)$, both $\psi_L$ and $\Psi_L$ are trivial, so $W_\sigma(g)\varphi(e_n g)$ is the indicator function on $N_n(K)H_n^1(K)$.

Since $H_n^1(K)$ is compact and $N_n(K)$ is unipotent, we get $|\det g|_F=1$ for all $g\in N_n(K)H_n^1(K)$, and $Z(W,\varphi,s)=vol(N_n(K)\backslash N_n(K)H_n^1(K))$.
\end{proof}

We are now ready to prove the following result which relates Asai gamma factors over finite fields and level zero representations. The proof is similar to that of Theorem 4.1 in \cite{Y}, and we only sketch it.

\begin{proposition}
\label{Equality} Let $\sigma$ be a level zero representation of $GL_n(L)$ coming from a pair $(\lambda, \pi)$, and if $\pi$ is not distinguished. Then
\[
\gamma(\pi,\psi,As)=vol(N_n(K)\backslash N_n(K)H_n^1(K))\gamma(s,\sigma,\psi,As).
\]
\end{proposition}
\begin{proof}
Take $W_\sigma$ to be Paskunas-Stevens'Whittaker function $W_\sigma(g)$, and let $\varphi$ be the indicator function on $e_nH_n^1(K)$. Consider the functional equation defining $\gamma(s,\sigma,\psi,As)$ as in Theorem \ref{LocalAsai}. By Lemma \ref{LHS}, the left hand side of the functional equation is simply
\[
\gamma(s,\sigma,\psi,As)vol(N_n(K)\backslash N_n(K)H_n^1(K)).
\]
Use change of variable $g\to \omega_n {^t}g^{-1}$, we get
\begin{eqnarray*}
RHS&=&\int_{N_n(K)\backslash GL_n(K) } W_\sigma(g)\CF_\psi\varphi(e_1 {^t}g^{-1})|\det g|_K^{s-1}dg   \\
&=& \sum_{l\in \BZ} q^{-nl(s-1)}\omega_\sigma(\varpi)^l \int_{N_n(\fo_K)\backslash GL_n(\fo_K) } \CJ_\sigma (g)\CF_\psi\varphi(\varpi^{-l}e_1 {^t}g^{-1}  )dg
\end{eqnarray*}
since $Supp W_\sigma\bigcap GL_n(K) \subseteq \coprod_{l\in \BZ} \varpi^l N_n(K)GL_n(\fo_K)$. $\CF_\psi\varphi(\varpi^{-l}e_1 {^t}g^{-1}  )$ is calculated in the middle of page 9 in \cite{Y}, plug that formula in and we get
\begin{equation}
\begin{aligned}
RHS&= (\sum_{l<0} q^{-nl(s-1)}\omega_\sigma(\varpi)^l)\int_{N_n(\fo_K)\backslash GL_n(\fo_K)} \CJ_\sigma(g)dg \\
&+ \int_{N_n(\fo_K)\backslash GL_n(\fo_K)} \CJ_\sigma (g)\psi(e_1 {^t}g^{-1}e_n)dg \\
&= (\sum_{l<0} q^{-nl(s-1)}\omega_\sigma(\varpi)^l) \sum_{N_n(F)\backslash GL_n(F)}\CB_{\pi,\psi}(g)+\sum_{N_n(F)\backslash GL_n(F)} \CB_{\pi,\psi}(g)\psi(e_1 {^t}g^{-1}e_n)      \label{Eq.1}
\end{aligned}
\end{equation}
The sum $\sum_{N_n(F)\backslash GL_n(F)}\CB_{\pi,\psi}(g)$ is zero since $\pi$ is not distinguished. Then the result follows from Proposition \ref{Bessel}.
\end{proof}

As an application, following the method in \cite{Y}, we give another way to compute the value of Asai gamma factor over finite fields when the representation is distinguished.

\begin{proposition}
\label{GammaValue} Suppose $\pi$ is an irreducible cuspidal distinguished representation of $G$. Then
\[
\gamma(\pi,\psi,As )=-1,
\]
\end{proposition}

\begin{proof}
Let $\sigma$ be the level zero representation of $GL_n(L)$ for the pair $(\lambda,\pi)$, where $\lambda\in \BC^\times$. Note that $\omega_\sigma(\varpi)=\lambda$, and the contragrdient $\wt{\pi}$ is from $(\lambda^{-1},\wt{\sigma} )$ by Lemma 2.1 in \cite{Y}. From equation \ref{Eq.1}, we get
\[
vol(N_n(K)\backslash N_n(K)H_n^1(K))\gamma(s,\sigma,\psi,As)=\frac{q^{n(s-1)}}{\lambda-q^{n(s-1)}}\sum_{N_n(F)\backslash GL_n(F)} \CB_{\pi,\psi}(g) + \gamma(\pi, \psi,As).
\]

Now from Lemma \ref{Lfactor},
\[
L(s,\sigma, As)=(1-\lambda q^{-ns})^{-1}
\]
and
\[
L(s,\sigma,As)=(1-\lambda^{-1}q^{-ns})^{-1}
\]
Therefore
\begin{equation}
\begin{aligned}
\epsilon(s,\sigma,\psi,As) &=\gamma(s,\sigma,\psi,As)\frac{L(s,\sigma,As)}{L(1-s, \wt{\sigma},As)}  \\
&=\frac{1}{vol(N_n(K)\backslash N_n(K)H_n^1(K))} \frac{\lambda^{-1}q^{-n}[c_1-\gamma(\pi,\psi,As)]q^{ns}+\gamma(\pi,\psi,As) }{q^{ns}-\lambda}\cdot q^{ns}, \label{Eq.2}
\end{aligned}
\end{equation}
where $c_1=\sum_{N_n(F)\backslash GL_n(F)} \CB_{\pi,\psi}(g)=q^n-1$. By Theorem 3 in \cite{K},
\[
\epsilon(s,\sigma,\psi,As)=c_2'q^{c_3s}
\]
where $c_2'\in \BC^\times, c_3\in \BZ$ depending on $\lambda$.

Then equality \ref{Eq.2} becomes
\begin{align}
q^{2ns}q^{-n}(c_1-\gamma(\pi,\psi,As))+\lambda \gamma(\pi,\psi,As) q^{ns}=c_2\lambda q^{(c_3+n)s}-c_2\lambda^2 q^{c_3s}.  \label{Eq.3}
\end{align}
where $c_2=vol(N_n(K)\backslash N_n(K)H_n^1(K))\cdot c_2'$. We note that $c_1,c_2,c_3$ and $\gamma(\pi,\psi,As)$ are independent of $s$. By comparing the powers of $q^s$ and the corresponding coefficients, we get
\begin{align*}
2n &= c_3+n \\
q^{-n}(c_1-\gamma(\pi,\psi,As))&= c_2\lambda \\
\lambda \gamma(\pi,\psi,As) &=-c_2\lambda^2.
\end{align*}
We then find
\begin{align*}
c_3 &= n \\
\gamma(\pi,\psi,As)&= \frac{c_1}{1-q^n}=-1 \\
\gamma(\pi,\psi,As) &=-c_2\lambda.
\end{align*}
Now the conclusion follows.
\end{proof}

\end{document}